\documentclass{amsproc}
\usepackage{amsfonts}

\theoremstyle{plain}

\newtheorem{lemma}{Lemma}

\newtheorem{problem}{Problem}

\newtheorem{theorem}{Theorem}
\numberwithin{equation}{section}

\begin{document}
\title[Measurable colorings of the real line]{Splitting necklaces and
measurable colorings of the real line}
\author{Noga Alon}
\address{Schools of Mathematics and Computer Science, Raymond and Beverly
Sacler Faculty of Exact Sciences, Tel Aviv University, Tel Aviv 69978,
Israel and IAS, Princeton, NJ 08540, USA }
\email{nogaa@tau.ac.il}
\author{Jaros\l aw Grytczuk}
\address{Theoretical Computer Science Department, Faculty of Mathematics and
Computer Science, Jagiellonian University, 30-387 Krak\'{o}w, Poland}
\email{grytczuk@tcs.uj.edu.pl}
\author{Micha\l\ Laso\'{n}}
\address{Faculty of Mathematics and Computer Science, Jagiellonian
University, 30-387 Krak\'{o}w, Poland}
\email{mlason@op.pl}
\author{Mateusz Micha\l ek}
\address{Faculty of Mathematics and Computer Science, Jagiellonian
University, 30-387 Krak\'{o}w, Poland}
\email{wajcha2@poczta.onet.pl}
\subjclass[2000]{Primary 05C38, 15A15; Secondary 05A15, 15A18}
\keywords{measurable coloring, splitting necklaces}

\begin{abstract}
A (continuous) \emph{necklace} is simply an interval of the real line
colored measurably with some number of colors. A well-known application of the
Borsuk-Ulam theorem asserts that every $k$-colored necklace can be fairly
split by at most $k$ cuts (from the resulting pieces one can form two
collections, each capturing the same measure of every color). Here we prove
that for every $k\geq 1$ there is a measurable $(k+3)$-coloring of the real
line such that no interval can be fairly split using at most $k$ cuts. In
particular, there is a measurable $4$-coloring of the real line in which no
two adjacent intervals have the same measure of every color. An analogous
problem for the integers was posed by Erd\H{o}s in 1961 and solved in the
affirmative in 1991 by Ker\"{a}nen. Curiously, in the discrete case the
desired coloring also uses four colors.
\end{abstract}

\maketitle

\section{Introduction}

In 1906 Thue \cite{Thue1}\ proved that there is a $3$-coloring of the
integers such that no two adjacent intervals are colored exactly the same.
This result has had lots of unexpected applications in distinct areas of
mathematics and theoretical computer science (cf. \cite{AlloucheAutom}, \cite%
{BeanAvoid}, \cite{Choffrut}, \cite{Lothaire1}). Many variations and
generalizations of this property were considered so far, specifically in
other combinatorial settings like Euclidean spaces \cite{BeanAvoid}, \cite%
{GrytczukArs}, \cite{GrytczukSliwa}, or graph colorings \cite{AlonNonrep}, 
\cite{AlonGrytczuk}, \cite{GrytczukIJMMS}.

In particular, in 1961 Erd\H{o}s \cite{ErdosStrong} (cf. \cite{CurrieOpen}, 
\cite{DekkingStrong}, \cite{GrytczukDM}) asked whether there is a $4$%
-coloring of the integers such that no two adjacent segments are identical,
even after arbitrary permutation of their terms. (In other words, there is
always a color whose number of occurrences in one segment differs from those in other segments.) It is not hard to check by hand that four colors are needed
for this property, but on the other hand, it is not obvious that any finite
number of colors is enough. This fact was established first by Evdokimov 
\cite{Evdokimov} who found a $25$-coloring with the desired property.
Another construction, provided by Pleasants \cite{Pleasants}, improved the
bound to $5$. That $4$ colors actually suffice was finally proved by Ker\"{a}%
nen \cite{Keranen}, with some verifications made by computer.

In the present paper we study a continuous variant of the Erd\H{o}s problem.
In particular, we prove that there exists a measurable $4$-coloring of the
real line such that no two adjacent segments contain equal measure of every
color. Actually our result is more general and relates to the continuous
version of the \emph{necklace splitting} problem. Let $f:\mathbb{R}%
\rightarrow \{1,2,\ldots ,k\}$ be a $k$-coloring of the real line such that
for every $i\in\{1,2,\ldots ,k\}$, the set $f^{-1}(i)$ of all points in color $i$
is Lebesgue measurable. A \emph{splitting} of size $r$ of an interval $[a,b]$
is a sequence of points $a=y_{0}<y_{1}<\ldots<y_{r}<y_{r+1}=b$ such that it
is possible to partition the resulting collection of intervals $F=\{[y_{i},y_{i+1}]:0\leq i\leq r\}$ into two disjoint subcollections $F_{1}$
and $F_{2}$, each capturing exactly half of the the total measure of every
color. The partition $F=F_{1}\cup F_{2}$ will be called a \emph{fair}
partition of $F$. For instance, in the continuous analog of the Erd\H{o}s
problem, intervals with splitting of size one are forbidden.

Goldberg and West \cite{GW} proved that every $k$-colored interval has a
splitting of size at most $k$ (see also \cite{AlonWest} for a short proof
using the Borsuk-Ulam theorem, and \cite{Matousek} for other applications of
the Borsuk-Ulam theorem in combinatorics). This result is clearly best
possible, as can be seen in a necklace where colors occupy consecutively full
intervals. Our result goes the other direction and provides an upper bound
for the number of colors in a general version of the Erd\H{o}s problem on
the line.

\begin{theorem}
\label{Main}For every $k\geq 1$ there is a $(k+3)$-coloring of the real line
such that no interval has a splitting of size at most $k$.
\end{theorem}

The proof is based on the Baire category theorem applied to the space of all
measurable colorings of $\mathbb{R}$, equipped with a suitable metric
(Section 2). By the same argument one may obtain other versions of the
result. For instance, one of these asserts that there is a $5$-coloring of $%
\mathbb{R}$ such that no two intervals (not necessarily adjacent) contain
the same measure of every color (Section 3). We do not know whether the
bound in Theorem \ref{Main} is optimal, even in the simplest case of $k=1$, though
one can show that two colors are no enough to avoid intervals with splitting
of size one.

\section{Proof of the main result}

Recall that a set in a metric space is \emph{nowhere dense} if the interior
of its closure is empty. A set is said to be of \emph{first category} if it
can be represented as a countable union of nowhere dense sets. In the proof
of Theorem \ref{Main} we apply the Baire category theorem in the following
version (cf. \cite{Oxtoby}).

\begin{theorem}[Baire] If $X$ is a complete metric space and $A$ is a set of first
category in $X$ then $X\setminus A$ is dense in $X$ (and in particular is
nonempty).
\end{theorem}

Our plan is to construct a suitable metric space of colorings and
demonstrate that the subset of \emph{bad} colorings is of first category.

\subsection{The setting}

Let $k$ be a fixed positive integer and let $\{1,2,\ldots ,k\}$ be the set
of colors. Let $f,g$ be two measurable $k$-colorings of $\mathbb{R}$. Let $n$
be another positive integer and consider the set 
\begin{equation*}
D_{n}(f,g)=\{x\in \lbrack -n,n]:f(x)\neq g(x)\}.
\end{equation*}%
Clearly $D_n(f,g)$ is measurable and we may define the normalized distance
between $f$ and $g$ on $[-n,n]$ by%
\begin{equation*}
d_{n}(f,g)=\frac{\lambda (D_{n}(f,g))}{n},
\end{equation*}%
where $\lambda(D)$ is the Lebesgue measure of $D$. Then we may define the
distance between any two measurable colorings $f$ and $g$ by%
\begin{equation*}
d(f,g)=\sum\limits_{n=1}^{\infty }\frac{d_{n}(f,g)}{2^{n+1}}.
\end{equation*}%
Identifying colorings whose distance is zero gives the metric space $%
\mathcal{M}$ of equivalence classes of all measurable $k$-colorings. Clearly
equivalent colorings preserve the anti-splitting properties we look for.

\begin{lemma}
\label{M-complete}$\mathcal{M}$ is a complete metric space.
\end{lemma}

The lemma is a straightforward generalization of the fact that sets of
finite measure in any measure space form a complete metric space, with
symmetric difference of sets as the distance function (see \cite%
{Oxtoby}).

Let $t\geq 1$ be a fixed integer. Let $\mathcal{D}_{t}$ be the subspace of $%
\mathcal{M}$ consisting of those $k$-colorings that avoid intervals having a
splitting of size at most $t$. Our task is to show that $\mathcal{D}_{t}$ is
not empty provided $t\leq k-3$. By \emph{granularity} of a splitting we mean
the length of the shortest subinterval in the splitting. For $r,n\geq 1$ let 
$\mathcal{B}_{n}^{(r)}$ be the set of those colorings from $\mathcal{M}$ for
which there exists at least one interval contained in $[-n,n]$ having a
splitting of size exactly $r$ and granularity at least $1/n$. Finally let 
\begin{equation*}
\mathcal{B}_{n}(t)=\bigcup\limits_{r=1}^{t}\mathcal{B}_{n}^{(r)}.
\end{equation*}
These are the bad colorings. Clearly we have 
\begin{equation*}
\mathcal{D}_{t}=\mathcal{M}\setminus \bigcup\limits_{n=1}^{\infty }\mathcal{B%
}_{n}(t).
\end{equation*}
So, our aim is to show that the sets $\mathcal{B}_{n}(t)$ are nowhere dense,
provided $k\geq t+3$, and hence the union $\bigcup\limits_{n=1}^{\infty }%
\mathcal{B}_{n}(t)$ is of first category.

\subsection{The sets $\mathcal{B}_{n}(t)$ are closed}

We show that each set $\mathcal{B}_{n}^{(r)}$ is a closed subset of $%
\mathcal{M}$. Since $\mathcal{B}_{n}(t)$ is a finite union of these sets, it
follows it is closed too. For any family $F$ of measurable subsets of $%
\mathbb{R}$ and a coloring $f$, we denote by $\lambda _{i}(f,F)$ the measure
of color $i$ in the union of all members of $F$.

\begin{lemma}
\label{Bn-closed}For every $r,n\geq 1$, the set $\mathcal{B}_{n}^{(r)}$ is a
closed subset of the space $\mathcal{M}$.
\end{lemma}

\begin{proof}
Fix $n,r\geq 1$. Let $f_{m}$ be an infinite sequence of colorings from $%
\mathcal{B}_{n}^{(r)}$ tending to the limit coloring $f$. Let $%
I^{(m)}=[a^{(m)},b^{(m)}]$ be a sequence of intervals in $[-n,n]$, each
having a splitting $a^{(m)}=y_{0}^{(m)}\leq y_{1}^{(m)}\leq \ldots \leq
y_{r}^{(m)}\leq y_{r+1}^{(m)}=b^{(m)}$ of granularity at least $1/n$. Let $%
F^{(m)}=\{[y_{i}^{(m)},y_{i+1}^{(m)}]:0\leq i\leq r\}$ be the resulting
family of intervals and let $F^{(m)}=F_{1}^{(m)}\cup F_{2}^{(m)}$ be the
related fair partition of $F^{(m)}$. Finally, let $\{0,1,\ldots
,r\}=A^{(m)}\cup B^{(m)}$ be the associated partition of indices, that is $%
F_{1}^{(m)}=\{[y_{i}^{(m)},y_{i+1}^{(m)}]:i\in A^{(m)}\}$ and $%
F_{2}^{(m)}=\{[y_{i}^{(m)},y_{i+1}^{(m)}]:i\in B^{(m)}\}$.

Since there are only finitely many index patterns $(A^{(m)},B^{(m)})$, one
of them must appear infinitely many times in the sequence. Hence, without
loss of generality, we may assume that all the $A^{(m)}$ and all the $B^{(m)}$ are
equal, say $A^{(m)}=A$ and $B^{(m)}=B$ for every $m\geq 1$. Since $[-n,n]$
is compact, there are subsequences of the $r+1$ sequences $y_{i}^{(m)}$, $%
0\leq i\leq r+1$, convergent to some points $y_{i}\in \lbrack -n,n]$,
respectively. Clearly the splitting $a=y_{0}\leq y_{1}\leq \ldots \leq
y_{r}\leq y_{r+1}=b$ has granularity at least $1/n$. Moreover, we claim that
the related fair partition $F=F_{1}\cup F_{2}$ of the family $%
F=\{[y_{i},y_{i+1}]:0\leq i\leq r\}$ has the same index pattern $(A,B)$.
Indeed, if this is not the case then there is a color $i$ such that%
\begin{equation*}
\left\vert \lambda _{i}(f,F_{1})-\lambda _{i}(f,F_{2})\right\vert
>\varepsilon ,
\end{equation*}%
for some $\varepsilon >0$. Taking $m$ large enough we can make the distances 
$d(f_{m},f)$ arbitrarily small, so that%
\begin{equation*}
\left\vert \lambda _{i}(f_{m},F_{1})-\lambda _{i}(f_{m},F_{2})\right\vert
>\varepsilon _{1}
\end{equation*}%
for some $\varepsilon _{1}>0$. Now, for sufficiently large $m$, the
symmetric difference between the unions of intervals in $F_{1}$ and $%
F_{1}^{(m)}$ can also be made arbitrarily small, hence%
\begin{equation*}
\left\vert \lambda _{i}(f_{m},F_{1}^{(m)})-\lambda
_{i}(f_{m},F_{2}^{(m)})\right\vert >\varepsilon _{2},
\end{equation*}%
for some $\varepsilon _{2}>0$. This contradicts the assumption that $%
F^{(m)}=F_{1}^{(m)}\cup F_{2}^{(m)}$ is a fair partition. Consequently, the
limit coloring $f$ must be in $\mathcal{B}_{n}^{(r)}$.
\end{proof}

\subsection{The sets $\mathcal{B}_{n}(t)$ have empty interiors}

Next we prove that each $\mathcal{B}_{n}(t)$ has empty interior provided the
number of colors $k$ is at least $t+3$. For that purpose, let us call $f$ an 
\emph{interval }coloring on $[-n,n]$ if there is a partition of the segment $%
[-n,n]$ into some number of (half-open) intervals of equal length, each
filled with only one color. Let $\mathcal{I}_{n}$ denote the set of all
colorings from $\mathcal{M}$ that are interval colorings on $[-n,n]$.

\begin{lemma}
\label{IntervalColoring}Let $f\in \mathcal{M}$ be a coloring. Then for every 
$\varepsilon >0$ and $n\in \mathbb{N}$ there exists a coloring $g\in 
\mathcal{I}_{n}$ such that $d(f,g)<\varepsilon $.
\end{lemma}

\begin{proof}
Let $C_{i}=f^{-1}(i)\cap \lbrack -n,n]$ and let $C_{i}^{\ast }\subset
\lbrack -n,n]$ be a finite union of intervals such that 
\begin{equation*}
\lambda ((C_{i}^{\ast }\setminus C_{i})\cup (C_{i}\setminus C_{i}^{\ast }))<%
\frac{\varepsilon }{2k^{2}}
\end{equation*}
for each $i=1,2,\ldots ,k$. Define a coloring $h$ so that for each $%
i=1,2,\ldots ,k$, the set $C_{i}^{\ast }\setminus (C_{1}^{\ast }\cup \ldots
\cup C_{i-1}^{\ast })$ is filled with color $i$, the rest of interval $[-n,n]
$ is filled with any of these colors, and $h$ agrees with $f$ everywhere
outside $[-n,n]$. Then $d(f,h)<\varepsilon /2$ and clearly each $%
h^{-1}(i)\cap \lbrack -n,n]$ is a finite union of intervals. Let $%
A_{1},\ldots ,A_{t}$ be the whole family of these intervals. Now split the
interval $[-n,n]$ into $N\geq 8tn/\varepsilon $ intervals $B_{1},\ldots
,B_{N}$ of equal length and define a new coloring $g$ so that $%
g(B_{i})=h(A_{j})$ if $B_{i}\subset A_{j}$, and $g(B_{i})$ is any color
otherwise. Hence, $g$ differs from $h$ on at most $2t$ intervals of total
length $2t2n/N\leq \varepsilon /2$ and we get $d(f,g)\leq \varepsilon $.
\end{proof}

\begin{lemma}
\label{EmptyInt}If $k\geq t+3$ then each set $\mathcal{B}_{n}(t)$ has empty
interior.
\end{lemma}

\begin{proof}
Let $f\in \mathcal{B}_{n}(t)$ be any bad coloring. Let $U(f,\varepsilon )$
be the open $\varepsilon $-neighborhood of $f$ in the space $\mathcal{M}$.
Assume the assertion of the lemma is false: there is some $\varepsilon >0$
for which $U(f,\varepsilon )\subset \mathcal{B}_{n}(t)$. By Lemma \ref%
{IntervalColoring} there is a coloring $g\in \mathcal{I}_{n}$ such that $%
d(f,g)<\varepsilon /2$, so $U(g,\varepsilon /2)\subset \mathcal{B}_{n}(t)$.
The idea is to modify slightly the interval coloring $g$ so that the new
coloring will be still close to $g$, but there will be no intervals inside $%
[-n,n]$ possessing a splitting of size at most $t$ and granularity at least $%
1/n $. Without loss of generality we may assume that there are equally spaced points $-n=a_{0}<a_{1}<\ldots <a_{N}=n$ in the
interval $[-n,n]$ such that each interval $A_{i}=[a_{i-1},a_{i})$ is filled
with a unique color in the interval coloring $g$. Let $\delta >0$ be a real
number satisfying 
\begin{equation*}
\delta <\min \left\{ \frac{\varepsilon }{2N},\frac{n}{N^{2}}\right\} .
\end{equation*}%
Consider another collection of intervals $S_{i}=[a_{i-1},b_{i})$, with $%
a_{i-1}<b_{i}<a_{i}$, such that $\lambda (S_{i})=\delta $ for all $%
i=1,2,\ldots ,N$.

Now split randomly each $S_{i}$ into $k$ subintervals $I_{i}^{(m)}$, $1\leq
m\leq k$. Then with probability one, for every fixed $m$, $1 \leq m \leq k$,
the set $\{\lambda (I_{i}^{(m)}):1\leq i\leq N \}$ is linearly independent
over the rationals, and in particular no two nonempty disjoint subsets of it
have the same total sum. Let $h_{i}$ be a $k$-coloring of $S_{i}$ defined by 
$h_{i}^{-1}(m)=I_{i}^{(m)}$ for every color $m\in \{1,2,\ldots ,k\}$.
Finally, let $h$ be a coloring such that $h=h_{i}$ on $S_{i}$, with $h=g$
everywhere outside the union of $S_{i}$.

First notice that%
\begin{equation*}
d(g,h)\leq \sum\limits_{i=1}^{N}\lambda (S_{i}) = \delta N<\frac{\varepsilon 
}{2},
\end{equation*}%
by the choice of $\delta $. Hence $h$ is in $\mathcal{B}_{n}(t)$. Let $%
[a,b]\subset \lbrack -n,n]$ be an interval with a splitting $%
a=y_{0}<y_{1}<\ldots <y_{r}<y_{r+1}=b$ of size $r\leq t$ and granularity at
least $1/n$. Let $F=F_{1}\cup F_{2}$ be the fair partition of the family $%
F=\{[y_{i},y_{i+1}]:0\leq i\leq r\}$, that is we have%
\begin{equation*}
\lambda _{j}(h,F_{1})=\lambda _{j}(h,F_{2})
\end{equation*}%
for every $1\leq j\leq k$. Since, $k\geq t+3>r+2$, there is a color $s\in
\{1,2,\ldots ,k\}$ that does not appear on any of the points $%
y_{0},y_{1},\ldots ,y_{r+1}$ in coloring $h$. Hence for every open interval $%
I$ filled with color $s$ in coloring $h$ and for every member $Y$ of the
family $F$, either $I$ is completely contained in $Y$, or $I$ and $Y$ are
disjoint.

We distinguish two types of intervals filled with color $s$: those of the
form $(b_{i},a_{i})$ (\emph{large }intervals), and those of the form $%
I_{i}^{(s)}$ (\emph{small} intervals). Let $l_{i}$ denote the number of
large intervals contained in the union of all members of $F_{i}$, $i=1,2$.
We claim that $l_{1}=l_{2}$. Indeed, suppose that this is not the case and
assume (without loss of generality) that $l_{1}>l_{2}$. Denote by $L=\frac{2n%
}{N}$ the common length of intervals $A_{i}=[a_{i-1},a_{i})$. Then%
\begin{equation*}
\lambda _{s}(h,F_{1})\geq l_{1}\left( L-\delta \right) =l_{1}L-l_{1}\delta
\geq l_{1}L-N\delta >l_{1}L-\frac{n}{N},
\end{equation*}%
by the choice of $\delta $. On the other hand,%
\begin{eqnarray*}
\lambda _{s}(h,F_{2}) &\leq &l_{2}L+N\delta \leq (l_{1}-1)L+N\delta
=l_{1}L-L+N\delta \\
&<&l_{1}L-\frac{2n}{N}+\frac{n}{N}=l_{1}L-\frac{n}{N},
\end{eqnarray*}%
again by the initial choice of $\delta $. This is a contradiction, so we
have $l_{1}=l_{2}$.

Since $\lambda _{s}(h,F_{1})=\lambda _{s}(h,F_{2})$, the sum of lengths of
small intervals of color $s$ in $F_{1}$ equals the sum of lengths of small
intervals of color $s$ in $F_{2}$. However this contradicts with the choice
of the numbers $\lambda (I_{i}^{(m)})$ as rationally independent. The proof
of the lemma is complete.
\end{proof}

\section{Generalizations and open problems}

In \cite{Alon} the necklace splitting theorem was generalized to fair
partitions into more than just two collections. A $q$\emph{-splitting} of
size $r$ of the necklace $[a,b]$ is a sequence $a=y_{0}<y_{1}<\ldots
<y_{r}<y_{r+1}=b$ such that it is possible to partition the resulting
collection of intervals $F=\{[y_{i},y_{i+1}]:0\leq i\leq r\}$ into $q$
disjoint subcollections $F_{1},F_{2},\ldots ,F_{q}$, each capturing exactly $%
1/q$ of the the total measure of every color. So, $F=F_{1}\cup F_{2}\cup
\ldots \cup F_{q}$ is a \emph{fair} partition of $F$ into $q$ parts. The
result of \cite{Alon} asserts that every $k$-colored necklace has a $q$%
-splitting of size at most $(q-1)k$, which is clearly best possible.

Notice that we may speak about fair partitions of any collection $F$ of
intervals on the line (not necessarily adjacent) whose interiors are
pairwise disjoint. Modifying slightly the proof of Theorem \ref{Main} one
may obtain the following more general result.

\begin{theorem}
For every $k\geq 4$ there exists a measurable $k$-coloring of the real line
with the following property: there is no family $F$ of closed intervals,
whose members have pairwise disjoint interiors and at most $k-1$ endpoints
in total, such that $F$ has a fair partition into $q$ parts, for any $q\geq
2 $.
\end{theorem}

For instance, there is a $5$-coloring of the real line such that there is no
pair of intervals (not necessarily adjacent), having the same measure of
every color (this solves an open problem posed in \cite{GrytczukSliwa}). It
is not known whether the constant $5$ is optimal for this property, but it
is not hard to show that two colors are not enough.

Let us go back at the end to the discrete case. By the discrete necklace
splitting theorem (see \cite{GW}, \cite{AlonWest}, \cite{Matousek}), any $k$%
-colored necklace has a splitting of size at most $k$ (provided the number
of beads in each color is even). It is not clear however if the following
discrete version of Theorem \ref{Main} holds.

\begin{problem}
Is it true that for every $k\geq 1$ there is a $(k+3)$-coloring of the
integers such that no segment has a splitting of size at most $k$?
\end{problem}


\begin{thebibliography}{99}
\bibitem{AlloucheAutom} J-P. Allouche, J. Shallit, Automatic sequences.
Theory, applications, generalizations, Cambridge University Press,
Cambridge, 2003.

\bibitem{Alon} N. Alon, Splitting necklaces, Advances in Math. 63 (1987),
247-253.

\bibitem{AlonNonrep} N. Alon, J. Grytczuk, M. Ha\l uszczak, O. Riordan,
Non-repetitive colorings of graphs, Random Struct. Alg. 21 (2002), 336-346.

\bibitem{AlonGrytczuk} N. Alon, J. Grytczuk, Breaking the rhythm on graphs,
Discrete Math. (2008)

\bibitem{AlonWest} N. Alon. D. West, The Borsuk-Ulam theorem and bisection
of necklaces, Proc. Amer. Math. Soc. 98 (1986), 623-628.

\bibitem{BeanAvoid} D.R. Bean, A. Ehrenfeucht, G.F. McNulty, Avoidable
patterns in strings of symbols, Pacific J. Math. 85 (1979), 261-294.

\bibitem{Brown} T.C. Brown, Is there a sequence on four symbols in which no
two adjacent segments are permutations of one another ? Amer. Math. Monthly
78 (1971), 886-888.

\bibitem{Choffrut} Ch. Choffrut, J. Karhum\"{a}ki, Combinatorics of Words,
in: Handbook of Formal Languages, G. Rozenberg, A. Salomaa (Eds.)
Springer-Verlag, Berlin Heidelberg, 1997, 329-438.

\bibitem{CurrieOpen} J.D. Currie, Open problems in pattern avoidance, Amer.
Math. Monthly 100 (1993), 790-793.

\bibitem{DekkingStrong} F.M. Dekking, Strongly non-repetitive sequences and
progression free sets, J. Combin. Theory Ser. A 27 (1979), 181-185.

\bibitem{ErdosStrong} P. Erd\H{o}s, Some unsolved problems, Magyar Tud.
Akad. Mat. Kutato. Int. Kozl. 6 (1961), 221-254.

\bibitem{Evdokimov} A. A. Evdokimov, Strongly asymmetric sequences generated
by finite number of symbols, Dokl. Akad. Nauk. SSSR 179 (1968), 1268-1271;
Soviet Math. Dokl. 9 (1968), 536-539.

\bibitem{GW} C. H. Goldberg, D. B. West, Bisection of circle colorings, SIAM
J. Algebraic Discrete Methods 6 (1985), 93--106.

\bibitem{GrytczukArs} J. Grytczuk, Pattern avoiding colorings of Euclidean
spaces, Ars Combin. 64 (2002), 109--116.

\bibitem{GrytczukSliwa} J. Grytczuk, W. \'{S}liwa, Non-repetitive colorings
of infinite sets, Discrete Math. 265 (2003), 365-373.

\bibitem{GrytczukIJMMS} J. Grytczuk, Nonrepetitive colorings of graphs - a
survey, Int. J. Math. Math. Sci. (2007)

\bibitem{GrytczukDM} J. Grytczuk, Thue type problems for graphs, points, and
numbers, Discrete Math. (2008)

\bibitem{Keranen} V. Ker\"{a}nen, Abelian squares are avoidable on 4
letters, Automata, Languages and Programming: Lecture notes in Computer
science 623 (1992) Springer-Verlag 4152.

\bibitem{Lothaire1} M. Lothaire, Combinatorics on Words, Addison-Wesley,
Reading MA, 1983.

\bibitem{Matousek} J. Matou\v{s}ek, Using the Borsuk-Ulam theorem,
Springer-Verlag, New York, 2003.

\bibitem{Oxtoby} J. C. Oxtoby, Measure and Category, Springer-Verlag, New
York, 1980.

\bibitem{Pleasants} P.A.B. Pleasants, Non-repetitive sequences, Proc.
Cambridge Philos. Soc. 68 (1970), 267-274.

\bibitem{Thue1} A. Thue, \"{U}ber unendliche Zeichenreichen, Norske Vid.
Selsk. Skr., I Mat. Nat. Kl., Christiania, 7 (1906), 1-22.

\bibitem{Thue2} A. Thue, \"{U}ber die gegenseitigen Lage gleicher Teile
gewisser Zeichenreihen, Norske Vid. Selsk. Skr., I Mat. Nat. Kl.,
Christiania, 1 (1912), 1-67.
\end{thebibliography}
\end{document}